\documentclass{article}
\usepackage{graphicx,amsmath,amsfonts,amssymb,amsthm,mathrsfs,bbm,stmaryrd,array, enumitem, array, caption, float, longtable}
\usepackage{hyperref}

\makeatletter
\def\th@plain{%
  \thm@notefont{}
  \itshape 
}
\def\th@definition{%
  \thm@notefont{}
  \normalfont 
}
\makeatother

\newtheorem{theorem}{Theorem}[section]
\newtheorem{proposition}[theorem]{Proposition}

\newtheorem{lemma}[theorem]{Lemma}
\newtheorem{corollary}[theorem]{Corollary}

\begin{document}
\title{The $p$-adic Valuations of M\"obius Duals of Lucas Sequences\footnotetext{\noindent Supported by National Natural Science Foundation of China, Project 12071421.}
\author{Tyler Ross\footnote{Corresponding author.}  \\ \small
\textit{School of Mathematical Sciences, Zhejiang University}\\
\small \textit{Hangzhou, Zhejiang, 310058, P.\ R.\ China} \\ \small \href{mailto:tylerxross@gmail.com}{\textit{tylerxross@gmail.com}}   
\and Zhongyan Shen \\ \small \textit{Department of Mathematics}, \\ \small \textit{Zhejiang International Studies University}\\
\small \textit{Hangzhou, Zhejiang, 310023, P.\ R.\ China} \\ \small \href{mailto:huanchenszyan@163.com}{\textit{huanchenszyan@163.com}}
\and Tianxin Cai \\ \small \textit{School of Mathematical Sciences, Zhejiang University}\\ \small \textit{Hangzhou, Zhejiang, 310058, P.\ R.\ China} \\ \small \href{mailto:txcai@zju.edu.cn}{\textit{txcai@zju.edu.cn}}}
\date{}
}

\maketitle

\begin{abstract}
        In this paper, we extend the $p$-adic valuations of the M\"obius duals of Lucas sequences, originally obtained by Carmichael for regular Lucas sequences to irregular Lucas sequences. We conclude with a brief observation about the relationships of these valuations to the existence of Wall-Sun-Sun primes.
\end{abstract}


\section{Introduction} \label{sec1}

For the purposes of this paper, the Lucas sequences
\begin{eqnarray*} 
    U(P,Q) = (U_n(P,Q))_{n \geq 0}, \\
    V(P,Q) = (V_n(P,Q))_{n \geq 0},
\end{eqnarray*}
of the first and second kind respectively, in parameters $P, Q \in \mathbb{Z} \setminus(0)$, are the second order linear recurrence integer sequences given by the Binet forms
\begin{equation*}
    U_n = \frac{\alpha^n - \beta^n}{\alpha - \beta}, ~V_n = \alpha^n + \beta^n    
\end{equation*}
where
\begin{equation*}
    \alpha = \frac{P + \sqrt{D}}{2}, ~\beta = \frac{P-\sqrt{D}}{2}   
\end{equation*}
are the roots of the characteristic polynomial $X^2-PX+Q \in \mathbb{Z}[X]$ with nonzero discriminant $D=P^2-4Q$. We assume moreover that $U$, $V$ are \textit{nondegenerate}, by which we mean that both $U_n, V_n \neq 0$ for all $n \geq 1$, or equivalently that $\alpha/\beta$ is not a root of unity. We do not, on the other hand, require that $U, V$ be \textit{regular} In other words, we allow for the possibility that $(P, Q) > 1$. When $P = 1$, $Q = -1$, we get the familiar Fibonacci numbers and Lucas numbers,
\begin{equation*}
    F = (F_n)_{n \geq 0} = (U_n(1,-1))_{n \geq 0} ~,~ L = (L_n)_{n \geq 0} = (V_n(1,-1))_{n \geq 0}.
\end{equation*}
In the following, we suppress all instances of the parameters $P$, $Q$ when these are taken to be arbitrary but fixed.

In order to study the divisibility properties of regular Lucas sequences, Carmichael worked with the sequences $M(P,Q) =M = (M_n)_{n \geq 1}$, given by the homogenized cyclotomic polynomials
\begin{equation*}
    M_n = \beta^{\varphi(n)}\Phi_n(\alpha/\beta),
\end{equation*}
where $\varphi: \mathbb{N} \longrightarrow \mathbb{N}$ is Euler's totient function and $\Phi_n \in \mathbb{Z}[X]$ is the $n$-th cyclotomic polynomial. Sequences of this form are sometimes referred to as Sylvester sequences (see \cite{Car}, \cite{Syl}). 

A straightforward calculation shows that $M_1 = \alpha - \beta$, and
\begin{equation*}
    U_n = \prod_{\substack{d \mid n \\ d > 1}}M_d
\end{equation*}
for $n > 1$. If $U$ is nondegenerate, then it follows by M\"obius inversion that
\begin{equation*}
    M_n = \prod_{d \mid n} U_d ^{\mu(n/d)}
\end{equation*}
for $n > 1$, where $\mu: \mathbb{N} \longrightarrow \{-1,0,1\}$ is the M\"obius $\mu$-function. 

More generally, if
\begin{equation*}
    A = n \longmapsto A_n: \mathbb{N} \longrightarrow S    
\end{equation*}
is any sequence taking values in a multiplicative subset $S$ of a commutative ring $\mathcal{R}$, we define the \textit{M\"obius dual}
\begin{equation*}
    M^A = n \longmapsto M^A_n: \mathbb{N} \longrightarrow S^{-1}\mathcal{R}
\end{equation*}
to be the sequence
\begin{equation*}
    M^{A}_{n} = \prod_{d \mid n} A_{d}^{\mu(n/d)};
\end{equation*}
equivalently (by M\"obius inversion), $M^A$ is uniquely determined by the relations
\begin{equation*}
    A_n = \prod_{d \mid n}M^A_d
\end{equation*}
for all $n \geq 1$.

When $U$ is a nondegenerate Lucas sequence, the sequences $M$ and $M^U$ are related by
\begin{equation*}
    \begin{cases}
        M_1 &= (\alpha-\beta)M_1^U, \\
        M_n &= M^U_n, \text{ if } n > 1.
    \end{cases}
\end{equation*}

It is clear that $M^A$ does not, in general, take values in integers for arbitrary integer sequences $A$; it turns out, however, that for any pair $U$, $V$ of Lucas sequences, $M^U$ is always an integer sequence, while $M^V_n \in\mathbb{Z}$ for all odd $n \geq 1$ and at most finitely many even $n$.

In this paper, we gather for reference several basic results about the sequences $M^U, M^V$; in particular, we extend the $p$-adic valuations for the sequences $M^U$ obtained by Carmichael under the hypothesis that $U$ is regular to irregular Lucas sequences. In a forthcoming paper, the authors make use of these results to obtain some congruences for both the sequences $M^U, M^V$, as well as the corresponding Lucas sequences, and to derive constraints on the entry point behavior of primes in Lucas sequences. We conclude with a brief observation relating these valuations to the existence of Wall-Sun-Sun primes.

\section{Results} \label{sec2}

We first derive a simple but useful doubling formula for the M\"obius dual sequences.

\begin{proposition}[Doubling formula] \label{dbl-thm}
For $n \geq 1$,
\begin{equation*} \label{dbleq0}
    M^{U}_{2n} = 
    \begin{cases}
        
        M^{V}_{n}, &\text{ \emph{if} } n \text{ \emph{is odd}.} \\
        M^{V}_{n}M^{U}_{n}, &\text{ \emph{if} } n \text{ \emph{is even},}
    \end{cases}
\end{equation*}
\end{proposition}

The next result gives the $p$-adic valuations of the numbers $M^U_n$ for all primes $p$, and all $n \geq 1$; in light of Proposition \ref{dbl-thm}, this determines also the $p$-adic valuations of the numbers $M^V_n$ $(n \geq 1)$. When $U$ is regular, in which case the first condition $p \nmid (P,Q)$ in the theorem below is automatically satisfied, these valuations agree with the analysis in Carmichael's original treatment of the subject (see \cite{Car}).

For $p$ prime, we write
\begin{equation*}
    z_U(p) = \text{min}(n \geq 1:p \mid U_n)
\end{equation*}
for the \textit{entry point}, or \textit{rank of apparition}, of $p$ in $U$. It is easy to verify that this number always exists, except in the case that $p \nmid P$ and $p \mid Q$ (see Proposition \ref{vpu-prop}). When $n=z_U(p)$, we say that $p$ is a \textit{characteristic factor} of $U_n$.

It is also convenient to introduce the notation
\begin{equation*}
    \partial_p(m)=m/p^{v_p(m)}
\end{equation*}
for the $p$-free part of $m \in \mathbb{Z} \setminus (0)$.

\begin{theorem} \label{vpM-thm}
    If $p \nmid (P,Q)$, we have the following cases.
    \begin{enumerate}[label = (\alph*)]
    \item \label{vpM0} If $p \mid Q$, then $v_p(M^U_n)=0$ for all $n \geq 1$.
    \item \label{vpM1} If $p \mid D$, then
    \begin{equation*} \label{vpMeq1}
        v_p(M^U_n) =
        \begin{cases}
         v_p(U_p), &\text{\emph{if} } n = p, \\
         1, &\text{\emph{if} } n = p^k,~ k >1, \\
         0, &\text{\emph{otherwise}}.
        \end{cases}
    \end{equation*}
    \item \label{vpM2} If $p \nmid QD$, then
    \begin{equation*}
        v_p(M^U_n) =
        \begin{cases}
            v_p(U_{z_U(p)}), &\text{\emph{if }} n=z_U(p), \\
            v_p(U_{pz_U(p)})-v_p(U_{z_U(p)}), &\text{\emph{if }} n = pz_U(p), \\
            1, &\text{\emph{if }} n =p^kz_U(p),~k >1, \\
            0, &\text{\emph{otherwise.}}
        \end{cases}
    \end{equation*}
    \end{enumerate}
    If $p \mid (P, Q)$, then we have the following cases.
    \begin{enumerate} [label = (\alph*)] \setcounter{enumi}{3}
        \item \label{vpm3} If $v_p(Q) \geq 2v_p(P)$, then
        \begin{equation*}
            v_p(M^U_n)=
            \begin{cases}
                0, &\text{\emph{if} } n = 1, \\
                \varphi(n)v_p(P)+v_p(M^{U^{(p)}}_n), &\text{\emph{if }} n > 1,
            \end{cases}
        \end{equation*}
        where $U^{(p)}=U(\partial_p(P),\partial_p(Q))$.
        \item \label{vpm4} If $2v_p(P) > v_p(Q)$, then 
        \begin{equation*}
            v_p(M^U_n) =
            \begin{cases}
                v_p(P), &\text{\emph{if n = 2}},\\
                \left(\varphi(n)/2\right)v_p(Q)+1, &\text{\emph{if }} n = 2p^k, ~ p \text{\emph{ prime }},~k\geq 1, \\
                \lfloor \varphi(n)/2 \rfloor v_p(Q), &\text{\emph{otherwise,}}
            \end{cases}
        \end{equation*}
        unless $2v_p(P)=v_p(Q)+1$, $p = 2$ or $3$, $n = 2p$ in which case
        \begin{align*}
            v_p(M^U_{2p})=v_p(Q)+1+v_p(\partial_p(P)^2-\partial_p(Q)). 
        \end{align*}
    \end{enumerate}
\end{theorem}

\begin{theorem} \label{int-thm}
    For any pair of Lucas sequences $U$, $V$, both $M^U_n, ~M^V_{2n-1} \in \mathbb{Z}$ for all $n \geq 1$, and $M^V_{2n} \in \mathbb{Z}$ for at most finitely many $n \geq 1$. In particular, if $U$ is regular and $M^V_{2n} \in \mathbb{Z}$, then $n \leq 6$, if $D > 0$,  and $n \leq 15$, if $D < 0$.
\end{theorem}

\section{Auxiliary Results} \label{sec3}
In this section, we gather some auxiliary results that we will need for the proofs of the main results. We make use without citation of the following facts from elementary number theory:
\begin{align*}
    &\sum_{d \mid n}\mu(d) =
    \begin{cases}
        1, &\text{if } n = 1, \\
        0, &\text{if } n > 1, 
    \end{cases} \\
    &\sum_{d \mid n}\mu(n/d)d = \varphi(n),
\end{align*}
for all $n \geq 1$.

Next, we recall in Lemma \ref{dbl-lem} and Proposition \ref{vpu-prop} a few basic facts concerning Lucas sequences.

\begin{lemma}[Doubling formula, \cite{Ball}, \cite{Luc}] \label{dbl-lem} For any integer $n \geq 1$,
    \begin{align*}
        \label{dbleq1} U_{2n} &= V_nU_n.
    \end{align*}
\end{lemma}

The calculation of the $p$-adic valuation for the sequences $M^U$ relies on the corresponding $p$-adic valuations for the sequences $U$. The fully general case, allowing for the possibility that $(P,Q) > 1$ was sorted out by Ballot (\cite{Ball1}, \cite{Ball}). The situation when $U$ is regular appears in more or less complete form scattered across Lucas's original treatment (\cite{Luc}) of the subject, and seems to have been rederived in various guises many times over since. Particularly concise formulations, after which the following proposition is modeled, were obtained by Sanna (\cite{Sanna}) for general Lucas sequences, and Lengyel (\cite{Leng}) in the special case $F=U(1,-1)$ of the Fibonacci numbers.

\begin{proposition}[Laws of appearance and repetition, \cite{Ball1}, \cite{Ball}, \cite{Luc}, \cite{Sanna}] \label{vpu-prop}
    Let $U$ be any Lucas sequence, $p$ a prime. If $p \nmid (P,Q)$, then we have the following cases. 
\begin{enumerate}[label = (\alph*)]
    \item \label{vpu0} If $p \mid Q$, then $v_p(U_n)=0$ for all $n \geq 1$.
    \end{enumerate}
Otherwise, if $p\nmid Q$, then $z_U(p)$ exists, and $z_U(p) \mid p -\left(\frac{D}{p} \right)$, unless $p=2$ does not divide $D$, in which case $z_U(p)=3$; in particular, $p \mid z_U(p)$ if and only if $p = z_U(p)$ if and only if $p \mid D$. Moreover, $v_p(U_{pz_U(p)}) \geq v_p(U_{z_U(p)}) + 1$, with equality if $p > 2$. We have the following valuations.
    \begin{enumerate}[label = (\alph*)] \setcounter{enumi}{1}
    \item \label{vpu1} If $p \mid D$, then
    \begin{equation*}
    v_p(U_n) =
    \begin{cases}
        v_p(U_p)+v_p(n)-1, &\text{\emph{if }} p \mid n,\\
        0, &\text{\emph{if }} p \nmid n.
    \end{cases}
    \end{equation*}
    \item \label{vpu2} If $p \nmid D$, then
    \begin{equation*}
    v_p(U_n) =
    \begin{cases}
        v_p(U_{pz_U(p)})+v_p(n)-1, &\text{\emph{if }} z_U(p) \mid n,~ p \mid n,\\
        v_p(U_{z_U(p)}), &\text{\emph{if }} z_U(p) \mid n,~ p \nmid n, \\
        0, &\text{\emph{if }} z_U(p) \nmid n.
    \end{cases}
    \end{equation*}
\end{enumerate}
If $p \mid (P, Q)$, then $p \mid U_n$ for every $n \geq 2$. We have the following cases.
\begin{enumerate}[label = (\alph*)] \setcounter{enumi}{3}
    \item If $v_p(Q) \geq 2v_p(P)$, then
    \begin{equation*}
        v_p(U_n)=(n-1)v_p(P)+v_p(U^{(p)}_n)
    \end{equation*}
    for all $n \geq 1$, where $U^{(p)}=U(\partial_p(P),\partial_p(Q))$.
    \item If $2v_p(P) > v_p(Q)$, then
    \begin{equation*}
        v_p(U_n)=
        \begin{cases}
        v_p(Q)\cdot \frac{n-1}{2}, &\text{\emph{if }} n \text{\emph{ is odd}}, \\
        v_p(Q)\cdot\frac{n}{2}+v_p(\frac{n}{2})+v_p(P)-v_p(Q)+h, &\text{\emph{if }} n \text{\emph{ is even}},
        \end{cases}
    \end{equation*}
    where
    \begin{equation*}
        h =
        \begin{cases}
           v_p(\partial_p(P)^2-\partial_p(Q)), &\text{\emph{if }} 2 \leq p \leq 3, ~v_p(Q)=2v_p(P)-1,~p\mid n, \\
           0, &\text{\emph{otherwise.}}
        \end{cases}
    \end{equation*}
\end{enumerate}
    
\end{proposition}

The integrality conditions for $M^V$ follow from Carmichael's theorem and its extension, almost a century later, to the complex case by Bilu, Hanrot, and Voutier. We omit some detailed case analyses from both theorems in favor of simplicity, similarly leaving out such case analysis from Theorem \ref{int-thm}.

\begin{theorem}[Carmichael's theorem, \cite{Car}] \label{cml-thm1}
    If $U$ is nondegenerate and regular, with $D > 0$, then $U_n$ has a characteristic factor for all $n > 12$.
\end{theorem}

\begin{theorem}[\cite{Bilu}] \label{cml-thm2}
    If $U$ is nondegenerate and regular, with $D < 0$, then $U_n$ has a characteristic factor for all $n > 30$.
\end{theorem}

Both of these theorems and their proofs rely on the hypothesis that $U$ is regular, but in fact the finiteness result can be extended to irregular Lucas sequences, although it is no longer possible to give universal bounds on the largest index admitting no characteristic factors. This observation, recorded in the following proposition, does not seem to have been written down anywhere in full generality, but the proof given by Durst (\cite{Durst}) for the real case ($D > 0$) makes use of that condition only insofar as the corresponding situation for regular sequences with $D < 0$ was still totally unresolved at that time (see also \cite{Bilu}, \cite{Lekk}, \cite{Ward}).

\begin{proposition} \label{charfact-prop}
If $U$ is a nondegenerate Lucas sequence, then $U_n$ has a characteristic factor for all but finitely many $n \geq 1$.
\end{proposition}

\section{Proofs}

\begin{proof}[Proof of Theorem \ref{dbl-thm} \\]
Suppose first that $n$ is odd. Then by the doubling formula in Lemma \ref{dbl-lem} and the definition of the sequences $M^U$, $M^V$,
\begin{align*}
    M^U_{2n} &=\prod_{d \mid 2n} U_d^{\mu(2n/d)} = \prod_{d \mid n}(U_{2d}/U_d)^{\mu(n/d)} = \prod_{d \mid n}V_d^{\mu(n/d)} = M^V_n.
\end{align*}
If $n$ is even, write $n = 2^kN$ where $N$ is odd. Then similarly
\begin{align*}
    M^U_{2n} &=\prod_{d \mid N} (U_{2^{k+1}d}/U_{2^k d})^{\mu(N/d)} = \prod_{d \mid N}V_{2^{k}d}^{\mu(N/d)} = M^V_n \prod_{d \mid N}V_{2^{k-1}d}^{\mu(N/d)} \\ &=M^V_n \prod_{d \mid N}(U_{2^ k  d}/U_{2^{k-1}d})^{\mu(N/d)} = M^V_nM^U_n.
\end{align*}
\end{proof}

\begin{proof}[Proof of Theorem \ref{vpM-thm}]

The proof in each case consists of involved but routine calculations using the valuations in Proposition \ref{vpu-prop} and the identity
\begin{equation*}
    v_p(M^U_n)=\sum_{d \mid n}\mu(n/d)v_p(U_d),
\end{equation*}
and, in particular,
\begin{equation*}
    v_p(M^U_{p^k N})= \sum_{d \mid N}\mu(N/d)(v_p(U_{p^k d})-v_p(U_{p^{k-1}d}))
\end{equation*}
for $p, k, N \geq 1$ with $p$ prime and $(p,N)=1$. We present the details in full only for the first nontrivial case, subsequently including only the points that require more careful consideration.

\ref{vpM0} Obvious.

\ref{vpM1} It is clear that $v_p(M^U_n)=0$ if $(p,n)=1$. Write $n = p^kN$ with $(p,N)=1$. If $k > 1$, then
\begin{equation*}
    v_p(U_{p^kd})-v_p(U_{p^{k-1}d})=1    
\end{equation*}
for all $d \mid N$, so
\begin{align*}
        v_p(M^U_n) = \sum_{d \mid N}\mu(d)=
        \begin{cases}
            1, &\text{if } N = 1, \\
            0, &\text{if } N > 1.
        \end{cases}
\end{align*}
If $k = 1$, then
\begin{equation*}
    v_p(U_{pd})-v_p(U_{d})=v_p(U_p)    
\end{equation*}
for all $d \mid N$, so \begin{align*}
        v_p(M^U_n) = \sum_{d \mid N}\mu(d)v_p(U_p)=
        \begin{cases}
            v_p(U_p), &\text{if } N = 1, \\
            0, &\text{if } N > 1.
        \end{cases}
\end{align*}

\ref{vpM2} It is clear that $v_p(M^U_n)=0$ if $z_U(p) \nmid n$. Suppose $n = z_U(p)N$ for some integer $N \geq 1$. Then
\begin{align*}
        v_p(M^U_n) = \sum_{d \mid z_U(p)N,~ z_U(p) \mid d}\mu(z_U(p)N/d)v_p(U_{d}) 
        = \sum_{d \mid N}\mu(N/d)v_p(U_{z_U(p)d}).
\end{align*}
Since $p \nmid D$, also $p \nmid z_U(p)$; if $p \nmid N$, then
\begin{equation*}
    v_p(U_{z_U(p)d})=v_p(U_{z_U(p)})    
\end{equation*}
for all $d\mid n$. If $p \mid N$, we rewrite this as $n = p^kz_U(p)N$ where $N \geq 1$, $(p,N)=1.$ Then
\begin{equation*}
    v_p(M^U_n) = \sum_{d \mid N}\mu(N/d)\left(v_p(U_{z_U(p)p^kd})-v_p(U_{z_U(p)p^{k-1}d}) \right),
\end{equation*}
and the remainder of the argument proceeds as in the previous case from the valuations in Proposition \ref{vpu-prop}.

\ref{vpm3} Evidently $v_p(M^U_1)=0$ for all primes $p$. For $n > 1$, after cancelling constant terms, we have
\begin{align*}
    v_p(M^U_n) &=v_p(P)\sum_{d \mid n}\mu(n/d)d+\sum_{d \mid n}\mu(n/d)v_p(U^{(p)}_d) \\
    &=\varphi(n)v_p(P)+v_p(M^{U^{(p)}}_n).
\end{align*}

\ref{vpm4} We always have the trivial cases $v_p(M^U_1)=0$, $v_p(M^U_2)=v_p(P)$. For $n > 1$ odd, ignoring constant terms,
\begin{equation*}
    v_p(M^U_n)=\frac{1}{2}v_p(Q)\sum_{d \mid n}\mu(n/d)d=(\varphi(n)/2)v_p(Q).
\end{equation*}
For even $n$, we assume first that $p > 3$ or $2v_p(P)-v_p(Q)>1$. If $n=2N$, $N$ odd, we have
\begin{equation*}
    v_p(U_{2d})-v_p(U_d) = \frac{1}{2}v_p(Q)d+v_p(d)+\text{a constant}
\end{equation*}
for all $d \mid N$, so
\begin{equation*}
    v_p(M^U_{n})=\frac{1}{2}\varphi(N)v_p(Q)+\sum_{d \mid N}\mu(N/d)v_p(d),
\end{equation*}
where
\begin{equation*}
    \sum_{d \mid N}\mu(N/d)v_p(d)=
    \begin{cases}
        1, &\text{if } n = p^k,~k\geq 1, \\
        0, &\text{otherwise}.
    \end{cases}
\end{equation*}
If $n=2^kN$ with $k \geq 2$, $N \geq 1$ odd, then
\begin{equation*}
    v_p(U_{2^k d})-v_p(U_{2^{k-1}d})=2^{k-2}v_p(Q)d+\text{a constant}
\end{equation*}
for all $d \mid N$.

Finally, if $p = 2 \text{ or } 3$, and $2v_p(P)=v_p(Q)+1$, we consider separately only the exceptional case $n=2p$. If $p = 2$, we have
\begin{align*}
    v_2(U_4) &=v_2(Q)v_2(P)+1+v_2(\partial_2(P)^2-\partial_2(Q)), \\
    v_2(U_2) &= v_2(P),
\end{align*}
so
\begin{align*}
    v_2(M^U_4) = v_2(U_4)-v_2(U_2)=v_2(Q)+1+v_2(\partial_2(P)^2-\partial_2(Q)).
\end{align*}
If $p = 3$, then
\begin{align*}
    v_3(U_6) &= 2v_3(Q)+v_3(P)+1+v_3(\partial_3(P)^2-\partial_3(Q)), \\
    v_3(U_3) &= v_3(Q), ~
    v_3(U_2) = v_3(P),
\end{align*}
so
\begin{equation*}
    v_3(M^U_6)=v_3(U_6)-v_3(U_3)-v_3(U_2) =v_3(Q)+1+v_3(\partial_3(P)^2-\partial_3(Q)).
\end{equation*}

\end{proof}

\begin{proof}[Proof of Theorem \ref{int-thm}]
    It is well known that $M^U_n \in \mathbb{Z}$ for all $n \geq 1$. One way to see this is to note that the valuations in Theorem \ref{vpM-thm} are always nonnegative; but this is overkill! Instead, it is enough to observe that the left-hand side of the identity 
    \begin{equation*}
        \prod_{d \mid n}U_d^{\mu(n/d)}=M^U_n=\beta^{\varphi(n)}\Phi_n(\alpha/\beta)
    \end{equation*}
    for $n > 1$ is rational, while the right-hand side is an algebraic integer. Obviously, $M^U_1=1 \in \mathbb{Z}$.
    
    Turning to the sequence $M^V$, if $n \geq 1$ is odd, then doubling formula in Proposition \ref{dbl-thm} shows that also
    \begin{equation*}
        M^V_n = M^U_{2n} \in \mathbb{Z}.    
    \end{equation*}
    Consider any even $n \geq 2$, and suppose $n = z_U(p)$ for some odd prime $p$. Then $v_p(M^U_n) > 0$ and $v_p(M^U_{2n}) = 0$ by Theorem \ref{vpM-thm}, so, again by the doubling formula,
    \begin{equation*}
        M^V_n = M^U_{2n}/M^U_n \in \mathbb{Q} \setminus\mathbb{Z}.
    \end{equation*}
    Proposition \ref{charfact-prop} shows that this hypothesis holds for all but finitely many $n \geq 1$, and Theorems \ref{cml-thm1} and \ref{cml-thm2} establish bounds on the largest $n \geq 1$ at which it can fail if $U$ is regular.
\end{proof}

\section{Conclusion}

We conclude with a brief discussion of Wall-Sun-Sun primes; in fact, the observations in this section do not rely on valuations more general than those already obtained by Carmichael, but the rather elegant characterization of Wall-Sun-Sun primes below does not seem to have been mentioned elsewhere in the literature. Recall that a Wall-Sun-Sun prime is a prime number that a prime number satisfying $v_p(F_{z_F(p)}) > 1$ (see \cite{Sun1}). It is not known whether or not any such primes exist, although it has been established that there are no Wall-Sun-Sun primes smaller than $9.7 \times10^{14}$ (see \cite{Dor}). In light of the valuations in the previous section, we have the following equivalence.

\begin{corollary}\label{WSS-cor}
    The following two statements are equivalent.
    \begin{enumerate}[label = (\alph*)]
        \item There are no Wall-Sun-Sun primes.
        \item The numbers $M^F_n$ are squarefree for all $n \neq 6$.
    \end{enumerate}
\end{corollary}



\begin{thebibliography}{99}
\bibitem{Ball1} \label{Ball2} C. Ballot, The $p$-adic valuation of Lucas sequences when $p$ is a special prime, \textit{Fib. Quart.} \textbf{57} (2019), 265--275.
\bibitem{Ball} \label{Ball} C. Ballot and H. Williams, \textit{The Lucas Sequences, Theory and Applications}, Springer, 2023.
\bibitem{Bilu} \label{Bilu} Y. Bilu, G. Hanrot, and P.M. Voutier, Existence of primitive divisors of Lucas and Lehmer numbers, \textit{J. Reine Angew. Math.} \textbf{539} (2001), 75--122.
\bibitem{Dor} \label{Dor} F. Dorais and D. Klyve, A Wieferich prime search up to $6.7 \times 10^{15}$, \textit{J. Integer Seq.} \textbf{14} (2011), Article 11.9.2.
\bibitem{Durst} \label{Durts} L. K. Durst, Exceptional real Lucas sequences, \textit{Pacific J. Math.} \textbf{11} (2) (1961), 489--494.
\bibitem{Car}\label{Car} R. D. Carmichael, On the numerical factors of the arithmetic forms \(\alpha^n \pm \beta^n\), \textit{Ann. Math.} \textbf{15} (1-4) (1913), 30--70.
\bibitem{Lekk} \label{Lekk} C. G. Lekkerkerker, Prime factors of the elements of certain sequences of integers, \textit{Nederl. Akad. Wetensch. Proc. (Series A)} \textbf{56} (1953), 265--280.
\bibitem{Leng}\label{Leng} T. Lengyel, The order of the Fibonacci and Lucas numbers, \textit{Fib. Quart.} \textbf{33} (3) (2013), 234--239.
\bibitem{Luc} \label{Luc} E. Lucas, Th\'eorie des fonctions num\'eriques simplements p\'eriodiques, \textit{Amer. J. Math.} \textbf{1} (1878), 184--240, 289--321.
\bibitem{Sanna} \label{Sanna} C. Sanna, The $p$-adic valuation of Lucas sequences, \textit{Fib. Quart.} \textbf{54} (2) (2016), 118--124.
\bibitem{Sun1} \label{Sun1} Z. H. Sun and Z. W. Sun, Fibonacci numbers and Fermat's last theorem, \textit{Acta Arith.} \textbf{60} (4) (1992), 371--388.
\bibitem{Syl} \label{Syl} J. J. Sylvester, On the divisors of cylcotomic functions, \textit{Amer. J. Math.} \textbf{2} (1879), 357--381.
\bibitem{Ward} \label{Ward} M. Ward, The intrinsic divisors of Lehmer numbers, \textit{Ann. Math. (Second Series)} \textbf{62} (1955), 230--236.

\end{thebibliography}
\end{document}